\documentclass{amsart}
\usepackage{graphicx}
\usepackage{amsmath}
\usepackage{amssymb}
\usepackage{amsthm}
\usepackage{enumitem}
\usepackage{tikz-cd}

\newtheorem{theorem}{Theorem}[section]
\newtheorem{proposition}[theorem]{Proposition}
\newtheorem{lemma}[theorem]{Lemma}
\newtheorem{example}[theorem]{Example}
\newtheorem{corollary}[theorem]{Corollary}
\newtheorem{hypothesis}[theorem]{Hypothesis}

\begin{document}
\author{Sebastian Gwizdek} %Please write full names and surnames of all co-authors here.
\address{Faculty of Applied Mathematics, AGH University of Science and Technology,
al. Mickiewicza 30, 30-059, Krak\'ow, Poland}
\email{gwizdek@agh.edu.pl}
\title{Isometric direct limits of bidual Banach spaces}
\keywords{direct limit, inverse limit, Banach space, Banach algebra, uniform algebra, Corona theorem, Arens Product, Gleason part, representing  measure}
\subjclass[2020]{32A38, 32A65, 32A70, 46A13, 46J10, 46J15, 46M15, 46M40}

\begin{abstract}
    %------------------------- Please type your abstract here ------------------
 Sequences of $n$-th order bidual Banach spaces, called tower systems and their direct and inverse limits are considered.  Motivated by recent applications in corona problem, we introduce two functors: \textrm{Dir} and \textrm{Inv} assigning to Banach spaces (and to  bounded linear operators) some new Banach spaces and operators. Of particular interest is  the enormous  "tower space" built over the space of continuous functions. We  prove that  the action of \textrm{Dir} preserves direct sum decompositions. This functor preserves also  spectra of  operators, their Fredholmness and  compactness properties. An application of these functors to the problem of location of supports of representing measures for function algebras is outlined in the last section.
    %---------------------------------------------------------------------------
\end{abstract}
\maketitle

%------------------- Type contents of your article below --------------------

\section{Introduction}

Direct (or \emph{inductive}) limits were used with success  in many branches of mathematics. From  theory of distributions to quite recent works \cite{AW}, \cite{WF} to name a few, such limits were studied in various categories. Here we present a construction of two functors in $\normalfont\textbf{Ban}_1$, the category of Banach spaces with linear contractions as morphisms.   The first functor assigns to a given Banach space a direct limit of the appropriate tower system and to a bounded linear mapping a bounded operator acting between corresponding direct limits. The second functor produces inverse limits of such systems. In the case of uniform Banach algebras the first functor proved   useful in studying corona problem  for a special class of strongly starlike strictly pseudoconvex domains \cite{KR}.   This problem has a long history dating back to its 1941 formulation by Shizuo Kakutani and its solution  in the case of the unit disc in 1962 by Lennart Carleson. The problem remained open for regular domains in higher dimensional case. The recent solution in \cite{KR} established corona theorem for strictly pseudoconvex domains which are also strongly starlike. In \cite{SG} the starlikeness condition has even been eliminated. One of the key  results  in \cite{KR} concerning supports of representing measures was obtained after taking  direct limit of a sequence of consecutive biduals of a uniform algebra. This construction suggested the introduction of our functors in a broad framework of Banach spaces.

The paper is organized as follows. In  Section 2 we   construct the functors \textrm{Dir} and \textrm{Inv}, which turn out to be adjoint in the sense described in Theorem 2.9. 

The action of both functors preserves  direct sum decompsitions, as shown in Theorem 2.6 and Corollary 2.10. In Theorems 2.11 and 2.12 we show that these functors preserve spectra of  bounded operators and the property of being a Fredholm operator, respectively.   The final theorem of Section 2. shows the preservation both under  \textrm{Dir}  and under \textrm{Inv} of   compactness of operators under an additional hypothesis. Namely we assume that the underlying Banach space possesses the approximation property (which requires compact operators to be  norm limits of   finite rank operators). 

  Section 3.  contains an application of these functors to a recent result of \cite{KR} on supports of representing measures for certain uniform algebras.

\section{Tower Functors}

Whenever we speak of direct or inverse limits of Banach spaces, we mean limits in the category $\normalfont\textbf{Ban}_1$. This  category consists of Banach spaces as objects and contractive (i.e. of norm less than or equal to 1) linear mappings as morphisms. It is well known that each direct (resp. inverse) system in $\normalfont\textbf{Ban}_1$ has a direct (resp. inverse) limit (see \cite{Se}, \S 11.8.2).

  Let $\left\{X_i\colon i\in I \right\}$ be a family of objects of $\normalfont\textbf{Ban}_1$ indexed by  a directed set $\langle I,\le\rangle$ and assume that $f_{ij}\colon X_i\to X_j$ are morphisms for all $i\le j$ such that
\begin{enumerate}[label={\textup{(\roman*)}}, widest=iii, leftmargin=*]
    \item $f_{ii}$ is the identity of $X_i$,
    \item $f_{ik}=f_{jk}\circ f_{ij}$ for all $i\le j\le k$.
\end{enumerate}
Then the pair $\langle X_i,f_{ij}\rangle$ is called a \textit{direct system} over $I$.    Here we will consider only \emph{isometric direct systems}, which means that the all the maps $f_{ij}$ are isometries and $I=\mathbb N$ will be the set of natural numbers.

On the disjoint union $\bigsqcup_{i\in \mathbb N} X_i:= \bigcup_i(X_i\times \{i\})$ we define the equivalence relation $(x_i,i)\sim (x_j,j)$, if for $k=\max\{i,j\}$ we have $f_{ik}(x_i)=f_{jk}(x_j)$. (Here $x_i\in X_i$ and $x_j\in X_j$.) In particular, $(x_i,i)\sim (f_{ij}(x_i),j)$ for all $i\le j$. Sometimes we shall write $x_i\sim x_j$ instead of $(x_i,i)\sim (x_j,j)$ to simplify the notation. On the set $\bigsqcup_{i\in I} X_i\big/_\sim$ we consider a standard quotient vector space linear structure. If our system is isometric, the norm of   $[(x_i,i)]\in\bigsqcup_{i\in I} X_i\big/_\sim$ is defined by $\Vert [(x_i,i)]\Vert :=\Vert x_i\Vert$ and it does not depend on the choice of representative of the equivalence class.

The \textit{direct limit} of the isometric direct system $\langle X_i,f_{ij}\rangle$, denoted by $\displaystyle\lim_{\longrightarrow} X_i$, is the completion in the norm of $\bigsqcup_{i\in I} X_i\big/_\sim$. The canonical mappings $f_{i,\infty}\colon X_i\to\displaystyle\lim_{\longrightarrow} X_i$ sending each element to its equivalence class are isometric morphisms in $\normalfont\textbf{Ban}_1$ under the above operations.

 For a Banach space $X$ let $X^{*}$ denote its dual. We define $X_0:=X$,\dots, $X_{n+1}:=X_{n}^{**}$ for $n=0,1,2,\ldots$. The family $\left\{X_n\colon n\in\mathbb{N}\right\}$ of second duals of Banach space $X$ will be called a \textit{tower system}. For any "base space" $X$ we denote by $\kappa_{n,n+1}$ the canonical embedding of $X_n$ into its bidual $X_{n+1}$. For $n\le m$ we define $\kappa_{n,m}\colon X_n\to X_m$ by
\begin{equation*}
\kappa_{n,m}:=\kappa_{m-1,m}\circ \kappa_{m-2,m-1}\circ\ldots\circ \kappa_{n, n+1}, \ \textrm{if} \ n<m
\end{equation*}
and $\kappa_{n,n}$ is defined to be the identity mapping. Since each $\kappa_{n,m}$ is an isometry we obtain an isometric direct system $\langle X_n,\kappa_{n,m}\rangle$.

Given $Y$, a second Banach space, denote   by $Y_n$ the related tower system, with canonical embeddings  denoted also by $\kappa_{n,m}$ and $\kappa_{n,\infty}$.   For  a bounded linear operator $T\in B(X,Y)$ we denote by $T^*$ its adjoint, i.e. the mapping $T^*: Y^*\ni \varphi \mapsto \varphi\circ T \in X^*$. Define $\mathcal{J}_0(T):=T$, $\mathcal{J}_{n+1}(T):=\mathcal{J}_{n}(T)^{**}$ for $n=0,1,2,\ldots$, so  that each $\mathcal{J}_{n}(T)\in B(X_n, Y_n)$ with $\|\mathcal{J}_{n}(T)\|=\|T\|$. Define $\displaystyle\lim_{\longrightarrow}\mathcal{J}_{n}(T)\colon\lim_{\longrightarrow} X_n\to\lim_{\longrightarrow} Y_n$ first on   $\bigsqcup_{i\in I} X_i\big/_\sim$ by
\begin{equation}
\lim_{\longrightarrow} \mathcal{J}_{n}(T) \left([(x_n,n)]\right):=[(\mathcal{J}_{n}(T)(x_n),n)] \ \textrm{for} \ x_n\in X_n .
\end{equation}
and then extend it by continuity to
the whole space (denoting this extension by the same symbol $\displaystyle\lim_{\longrightarrow}\mathcal{J}_{n}(T)$).

Finally, we define the functor $\textrm{Dir}\colon\normalfont\textbf{Ban}\to\normalfont\textbf{Ban}$ as follows
\begin{equation*}
\textrm{Dir}\colon X\mapsto\lim_{\longrightarrow} X_n,
\end{equation*}
\begin{equation*}
\textrm{Dir}\colon T\mapsto\lim_{\longrightarrow} \mathcal{J}_{n}(T).
\end{equation*}

The following lemma shows that the mapping $\displaystyle\lim_{\longrightarrow} \mathcal{J}_{n}(T)$ is well defined.
\begin{lemma}
If $x_n\in X_n$ for some $n\ge 0$ then
\begin{equation*}
(\mathcal{J}_{n+1}(T)\circ\kappa_{n,n+1})(x_n)=(\kappa_{n,n+1}\circ \mathcal{J}_{n}(T))(x_n).
\end{equation*}
\end{lemma}
\begin{proof}
For any $\gamma\in Y_n^*$ we have
\begin{equation*}
\begin{split}
\langle (\kappa_{n,n+1}\circ \mathcal{J}_{n}(T))(x_n), \gamma \rangle = \langle \gamma, \mathcal{J}_{n}(T)(x_n) \rangle = \langle \mathcal{J}_{n}(T)^{*}(\gamma), x_n\rangle = \\
=\langle \kappa_{n,n+1}(x_n), \mathcal{J}_{n}(T)^{*}(\gamma) \rangle = \langle (\mathcal{J}_{n+1}(T)\circ\kappa_{n,n+1})(x_n), \gamma \rangle.
\end{split}
\end{equation*}
%Since $\gamma$ was arbitrary we obtain the result.
\end{proof}
\begin{proposition}
%The functor
 $\normalfont\textrm{Dir}$ is a covariant functor from $\normalfont\textbf{Ban}$ to $\normalfont\textbf{Ban}$.
\end{proposition}
\begin{proof}
First we check that $\displaystyle{ \lim_{\longrightarrow} }\mathcal{J}_{n}(T)$ is well defined. If $[(x_n,n)]= [(x_{n+1},n+1)]$, i.e. $x_n\sim x_{n+1}$, then $x_{n+1}=\kappa_{n,n+1}(x_n)$.  From Lemma 3.1 we obtain the equality
\begin{equation*}
\mathcal{J}_{n+1}(T)(x_{n+1})=(\mathcal{J}_{n+1}(T)\circ\kappa_{n,n+1})(x_n)=(\kappa_{n,n+1}\circ \mathcal{J}_{n}(T) )(x_n).
\end{equation*}
This implies $\mathcal{J}_{n+1}(T)(x_{n+1})\sim \mathcal{J}_{n}(T)(x_n)$. For arbitrary $n\le m$, if $x_n\sim x_m$ then $\kappa_{n,m-1}(x_n)\sim x_m$. Hence $\mathcal{J}_{m}(T)(x_m)\sim \mathcal{J}_{m-1}(T)(\kappa_{n,m-1}(x_n))$ from which we obtain the equality
\begin{equation*}
\mathcal{J}_{m}(T)(x_{m})=\kappa_{m-1,m}(\mathcal{J}_{m-1}(T)(\kappa_{n,m-1}(x_n))).
\end{equation*}
Now $\kappa_{n,m-1}(x_n)\sim\kappa_{n,m-2}(x_n)$ so that
\begin{equation*}
\mathcal{J}_{m-1}(T)(\kappa_{n,m-1}(x_n))\sim \mathcal{J}_{m-2}(T)(\kappa_{n,m-2}(x_n)).
\end{equation*}
Hence
\begin{equation*}
\mathcal{J}_{m}(T)(x_{m})=(\kappa_{m-1,m}\circ\kappa_{m-2,m-1})(\mathcal{J}_{m-2}(T)(\kappa_{n,m-2}(x_n))).
\end{equation*}
Repeating the above reasoning yields
\begin{equation*}
\begin{split}
\mathcal{J}_{m}(T)(x_{m})=(\kappa_{m-1,m}\circ\ldots\circ\kappa_{n+1,n+2})(\mathcal{J}_{n+1}(T)(\kappa_{n,n+1}(x_n)))= \\
=(\kappa_{m-1,m}\circ\ldots\circ\kappa_{n,n+1})(\mathcal{J}_{n}(T)(x_n))=\kappa_{n,m}(\mathcal{J}_{n}(T)(x_n)).
\end{split}
\end{equation*}
%The last but one equality follows from Lemma 3.1.
 Hence we obtain that $\mathcal{J}_{n}(T)(x_{n})\sim \mathcal{J}_{m}(T)(x_m)$.

Take any $T\in B(X,Y)$ and $S \in B(Y,Z)$.  For $x_n\in X_n$ we have
\begin{equation*}
\begin{split}
\lim_{\longrightarrow} (\mathcal{J}_{n}(S)\circ \mathcal{J}_{n}(T)) \left([(x_n,n)]\right)=[(\mathcal{J}_{n}(S)(\mathcal{J}_{n}(T)(x_n)),n)]= \\
=\lim_{\longrightarrow} \mathcal{J}_{n}(S) \left([(\mathcal{J}_{n}(T)(x_n),n)]\right)=\lim_{\longrightarrow} \mathcal{J}_{n}(S) \left(\lim_{\longrightarrow}\mathcal{J}_{n}(T)([(x_n,n)])\right).
\end{split}
\end{equation*}
Hence $\textrm{Dir}(S\circ T)=\textrm{Dir}(S)\circ \textrm{Dir}(T)$.

Clearly $\textrm{Dir}(T)$ is  bounded and linear. %We show now that $\textrm{Dir}(T)$ is a bounded operator. Indeed, for $x_n\in X_n$ we have
%\begin{equation*}
%\begin{split}
%\Vert \textrm{Dir}(T)\left([(x_n,n)]\right)\Vert=\Vert\left([(\mathcal{J}_{n}(T)(x_n),n)]\right)\Vert= \\ =\Vert \mathcal{J}_{n}(T)(x_n)\Vert\le\Vert \mathcal{J}_{n}(T)\Vert\cdot\Vert x_n\Vert=\Vert T\Vert\cdot\Vert x_n\Vert.
%%\end{equation*}
%It follows that $\Vert\textrm{Dir}(T)\Vert\le\Vert T\Vert$.

Finally, the functor $\textrm{Dir}$ preserves the identity mapping (by obvious calculations).% since for any $x_n\in X_n$ we have

%\begin{equation*}
%%\textrm{Dir}(\textrm{id}_X)\left([(x_n,n)]\right)=\lim_{\longrightarrow} \mathcal{J}_{n}(\textrm{id}_X)\left([(x_n,n)]\right)=\left[(\mathcal{J}_{n}(\textrm{id}_X)(x_n),n)\right]= \\ =\left[(\textrm{id}_{X_n}(x_n),n)\right]=[(x_n,n)]=\textrm{id}_{\textrm{Dir}(X)}([(x_n,n)]).
%\end{split}
%\end{equation*}
%The quotient space is dense in the entire set $\displaystyle\lim_{\longrightarrow} X_n$ hence the proof is complete.
\end{proof}
One can easily notice that the mapping $\textrm{Dir}\colon X\mapsto\displaystyle\lim_{\longrightarrow} X_n$ in general is not injective. Indeed, for any Banach space $X$ we have $\textrm{Dir}(X)=\textrm{Dir}(X^{**})$ up to isomorphism. In the case when the Banach space $X$ is reflexive we even have the equality $\textrm{Dir}(X)=X$. The following example shows that $\textrm{Dir}(X)$  can be small even for non-reflexivs $X$.

%Before we move to the further results let us consider two examples illustrating substantially different structure of direct limit of tower system. In the first example the obtained space will be "small" in a sense that the limit is a separable Banach space.

\begin{example}
\normalfont
Let $\mathfrak{J}$ denote the \textit{James space}. The space $\mathfrak{J}$ is a separable Banach space of codimension one in its bi-dual, i.e., $\dim\mathfrak{J}^{**}\big/\kappa_{0,1}(\mathfrak{J})=1$. Actually, $\mathfrak{J}$ is isometrically isomorphic to $\mathfrak{J}^{**}$. More information on this space can be found in \cite{M}. Since the subspace $ \kappa_{0,1}(\mathfrak{J})$ is complemented in $\mathfrak{J}^{**}$,    for some $e_1\in\mathfrak{J}^{**},\dots, e_n\in \mathfrak J_n$  the tower system $\left\{\mathfrak{J}_n\right\}$   of $\mathfrak{J}$ satisfies up to an isomorphism   the equalities
\begin{equation*}
\mathfrak{J}_1=\mathfrak{J}^{**}=\mathfrak{J}\oplus\mathbb{C}e_1,
\end{equation*}
\begin{equation*} \mathfrak{J}_2=\mathfrak{J}_1^{**}=\mathfrak{J}_1\oplus\mathbb{C}e_2=\mathfrak{J}\oplus\mathbb{C}e_1\oplus\mathbb{C}e_2,
\end{equation*}
%Repeating the above reasoning yields
\begin{equation*}
\mathfrak{J}_n=\mathfrak{J}\oplus\mathbb{C}e_1\oplus\ldots\oplus\mathbb{C}e_n.
\end{equation*}
Taking the direct limit we obtain  
\begin{equation*}
\textrm{Dir}(\mathfrak{J})=\mathfrak{J}\oplus\mathfrak{R},
\end{equation*}
where $\mathfrak{R}$ denotes the completion of linear span of vectors $\left\{e_k\colon k=1,\ldots, n\right\}$ in which the norms of vectors $\lambda_1e_1+\ldots +\lambda_n e_n$, $\lambda_k\in\mathbb{C}$, are calculated in the space $\mathfrak{J}_n$, $n\ge 1$. Hence $\textrm{Dir}(\mathfrak{J})$ is a separable space.
\end{example}

In the second example the tower space will be large, far from  separable.

\begin{example}
\normalfont
Let $K$ be a compact, infinite Hausdorff space with $\mathcal{F}_1$ --a maximal singular family of probabilistic regular Borel measures on $K$ (which can be arbitrarily chosen). Here singularity means that $\mu\perp\nu$ for any $\mu\neq \nu, \mu,\nu\in \mathcal F$. Let $\left\{\mathcal{C}_n\right\}$ denote the tower system of $C(K)$. Define $U_{\mathcal{F}_1}$ to be the disjoint topological union of the spectra $\Phi_\mu$ of the Banach algebras $L^{\infty}(K,\mu)$, ($\Phi_\mu$ are the sets of all non-zero linear and multiplicative functionals on $L^{\infty}(K,\mu)$).  Then each $\Phi_\mu$ is a compact and open subspace of $U_{\mathcal{F}_1}$ and by %the result of
 \cite[Theorem 5.4.4]{D} we have the equality
\begin{equation*}
    \mathcal{C}_1=C(K)^{**}=C(\beta U_{\mathcal{F}_1}),
\end{equation*}
where $\beta$ denotes the Stone-\v{C}ech compactification. Also for $n\ge 2$ we obtain %the relation
\begin{equation*}
    \mathcal{C}_n=C(\beta U_{\mathcal{F}_n}).
\end{equation*}
Here $\mathcal{F}_n$ is a maximal singular family of probabilistic regular borel measures on $\beta U_{\mathcal{F}_{n-1}}$ and containing $\kappa_{n-1,n}(\mathcal{F}_{n-1})$. The mapping $\kappa_{n-1,n}$ is an isometry, hence from the singularity relation $\mu\perp\nu$ it follows that $\kappa_{n-1,n}(\mu)\perp\kappa_{n-1,n}(\nu)$. We use the fact that $\mu\perp\nu$ if and only if $\Vert\mu\pm\nu\Vert=\Vert\mu\Vert +\Vert\nu\Vert$. Hence $\kappa_{n-1,n}(\mathcal{F}_{n-1})$ is a family of singular measures. Taking the direct limit we get the equality
\begin{equation*}
    \textrm{Dir}(C(K))=C\left(\beta \left(\bigsqcup U_{\mathcal{F}_n}\right)\right),
\end{equation*}
where $\bigsqcup U_{\mathcal{F}_n}$ denotes the topological disjoint union of the spaces $U_{\mathcal{F}_n}$ in which each $\Phi_\mu$ is a compact and open set for every $\mu\in\mathcal{F}_n$, $n\in\mathbb{N}$.
\end{example}

Our next theorem will show that the action of functor $\textrm{Dir}$ preserves finite direct sum decompositions. But first we need the following lemma. For $T\in B(X):= B(X,X)$ we denote by $\mathcal{N}(T)$ and $\mathcal{R}(T)$ its kernel and range respectively.

\begin{lemma}
If $P\in B(X)$ is a projection, then $P^{**}\in B(X^{**})$ is a projection onto $\mathcal{R}(P^{**})=\mathcal{R}(P)^{**}$ and such that  $\mathcal{N}(P^{**})=\mathcal{N}(P)^{**}$.
\end{lemma}
\begin{proof}
From the multiplicative property of taking the adjoint we see that $(P^*)^2=P^*$, i.e. $P^{*}\in B(X^{*})$ is a projection and so is $P^{**}\in B(X^{**})$.

It remains to show that $\mathcal{R}(P^{**})=\mathcal{R}(P)^{**}$ and $\mathcal{N}(P^{**})=\mathcal{N}(P)^{**}$. First we shall prove that $\kappa_{0,1}(\mathcal{R}(P))\subset\mathcal{R}(P^{**})$ and $\kappa_{0,1}(\mathcal{N}(P))\subset\mathcal{N}(P^{**})$. Take any $y\in\mathcal{R}(P)$ of the form $y=Px$ for some $x\in X$. Then for any $\phi\in X^{*}$ we obtain the equalites
\begin{equation*}
    \langle P^{**}(\kappa_{0,1}(x)) ,\phi\rangle=\langle\kappa_{0,1}(x), P^{*}\phi\rangle=\langle P^{*}\phi, x\rangle=\langle \phi, Px\rangle=\langle\phi, y\rangle=\langle\kappa_{0,1}(y), \phi\rangle.
\end{equation*}
Hence $\kappa_{0,1}(\mathcal{R}(P))\subset\mathcal{R}(P^{**})$.

Similarly, for any $x\in\mathcal{N}(P)$ and $\phi\in X^{*}$ we have
\begin{equation*}
    \langle P^{**}(\kappa_{0,1}(x)) ,\phi\rangle=\langle\kappa_{0,1}(x), P^{*}\phi\rangle=\langle P^{*}\phi, x\rangle=\langle \phi, Px\rangle=0
\end{equation*}
and $\kappa_{0,1}(\mathcal{N}(P))\subset\mathcal{N}(P^{**})$. The sets $\mathcal{R}(P^{**})$ and $\mathcal{N}(P^{**})$ are weak-* closed in $X^{**}$. From the Goldstine theorem canonical images of $\mathcal{R}(P)$ and $\mathcal{N}(P)$ are weak-* dense in $\mathcal{R}(P^{**})$ and $\mathcal{N}(P^{**})$. As a consequence we obtain the inclusions $\mathcal{R}(P)^{**}\subset\mathcal{R}(P^{**})$ and $\mathcal{N}(P)^{**}\subset\mathcal{N}(P^{**})$.

The operator $P$ is a projection so we have the decomposition $X=\mathcal{R}(P)\oplus\mathcal{N}(P)$. It follows that
\begin{equation*}
    X^{**}=\mathcal{R}(P)^{**}\oplus\mathcal{N}(P)^{**}=\mathcal{R}(P^{**})\oplus\mathcal{N}(P^{**}).
\end{equation*}
Take any $\phi\in\mathcal{R}(P^{**})$. Then $\phi$ has a unique decomposition of the form $\phi=\phi_1+\phi_2$, where $\phi_1\in\mathcal{R}(P)^{**}$ and $\phi_2\in\mathcal{N}(P)^{**}$. From the already proved inclusions we obtain $\phi_1\in\mathcal{R}(P^{**})$ and $\phi_2\in\mathcal{N}(P^{**})$. Hence $\phi_2=\phi-\phi_1\in\mathcal{R}(P^{**})$. However $\mathcal{R}(P^{**})\cap\mathcal{N}(P^{**})=\left\{0\right\}$, so that $\phi=\phi_1\in\mathcal{R}(P)^{**}$. This proves the inclusion $\mathcal{R}(P^{**})\subset\mathcal{R}(P)^{**}$. The proof that $\mathcal{N}(P^{**})\subset\mathcal{N}(P)^{**}$ is analogous.
\end{proof}

We are ready to prove our next theorem.

\begin{theorem}
Let $X$ be a Banach space. If $X=U\oplus V$ for some closed subspaces $U$, $V$ of $X$ then $\normalfont\textrm{Dir}(X)=\textrm{Dir}(U)\oplus\textrm{Dir}(V)$.
\end{theorem}
\begin{proof}
Let $P\in B(X)$ be a projection such that $\mathcal{R}(P)=U$ and $\mathcal{N}(P)=V$. Then $\textrm{Dir}(P)\in B(\textrm{Dir}(X))$ and
\begin{equation*}
    \textrm{Dir}(P)=\textrm{Dir}(P^2)=\textrm{Dir}(P)^2.
\end{equation*}
Hence $\textrm{Dir}(P)$ is a bounded projection. In order to prove that $\mathcal{R}(\textrm{Dir}(P))=\textrm{Dir}(U)$ and $\mathcal{N}(\textrm{Dir}(P))=\textrm{Dir}(V)$
  we  first  show that $\textrm{Dir}(U)\subset\mathcal{R}(\textrm{Dir}(P))$.
   
   Let $\left\{U_n\right\}$ denote the tower system of $U$. Take any $[(y_n,n)]\in\bigsqcup U_n\big/_\sim$. From the previous lemma we deduce the equalities
\begin{equation*}
    \bigsqcup U_n\big/_\sim=\bigsqcup \left(\textrm{Dir}(P)(X_n \big/_\sim)\right)=\textrm{Dir}(P)\left(\bigsqcup X_n\big/_\sim\right).
\end{equation*}
It follows that $[(y_n,n)]\in\mathcal{R}(\textrm{Dir}(P))$. The set $\bigsqcup U_n\big/_\sim$ is dense in $\textrm{Dir}(U)$ and the space $\mathcal{R}(\textrm{Dir}(P))$ is closed. Hence $\textrm{Dir}(U)\subset\mathcal{R}(\textrm{Dir}(P))$.

Conversely, from the continuity of $\textrm{Dir}(P)$ we have that
\begin{equation*}
    \mathcal{R}(\textrm{Dir}(P))=\textrm{Dir}(P)\left(\overline{\bigsqcup X_n\big/_\sim}\right)\subset\overline{\textrm{Dir}(P)\left(\bigsqcup X_n\big/_\sim\right)}=\textrm{Dir}(U).
\end{equation*}
Hence $\mathcal{R}(\textrm{Dir}(P))=\textrm{Dir}(U)$.

Now we show that $\textrm{Dir}(V)\subset\mathcal{N}(\textrm{Dir}(P))$. Let $\left\{V_n\right\}$ denote the tower system of $V$. Take any $[(x_n,n)]\in\bigsqcup V_n\big/_\sim$. Using the lemma preceding this theorem we obtain the equalities
\begin{equation*}
    \bigsqcup V_n\big/_\sim=\bigsqcup \left(\mathcal{N}(\textrm{Dir}(P)\cap\left(X_n \big/_\sim\right)\right)=\mathcal{N}(\textrm{Dir}(P))\cap\left(\bigsqcup X_n\big/_\sim\right).
\end{equation*}
As a consequence $[(x_n,n)]\in\mathcal{N}(\textrm{Dir}(P))$. The set $\bigsqcup V_n\big/_\sim$ is dense in $\textrm{Dir}(V)$ and the space $\mathcal{N}(\textrm{Dir}(P))$ is closed so we obtain the inclusion $\textrm{Dir}(V)\subset\mathcal{N}(\textrm{Dir}(P))$.

Conversely, using the continuity of $\textrm{Dir}(P)$ we get the relactions
\begin{equation*}
\begin{split}
    \mathcal{N}(\textrm{Dir}(P))=\textrm{Dir}(I-P)\left(\overline{\bigsqcup X_n\big/_\sim}\right)\subset\overline{\textrm{Dir}(I-P)\left(\bigsqcup X_n\big/_\sim\right)}= \\
    =\overline{\mathcal{N}(\textrm{Dir}(P))\cap\left(\bigsqcup X_n\big/_\sim\right)}=\overline{\bigsqcup V_n\big/_\sim}=\textrm{Dir}(V).
\end{split}
\end{equation*}
Hence $\mathcal{N}(\textrm{Dir}(P))=\textrm{Dir}(V)$ and the proof is finished.
\end{proof}

Let us now recall the definition of natural transformations of functors. Let $\mathcal{C}$ and $\mathcal{C'}$ be categories. A \textit{natural transformation} of functors $F$, $G\colon\mathcal{C}\to \mathcal{C'}$ is any family  $\eta$ of mappings: $\eta=(\eta_X)$ assigned to all    objects $X$ in $\mathcal{C}$, such that

\begin{enumerate}[label={\textup{(\roman*)}}, widest=iii, leftmargin=*]
\item Each $\eta_X\colon F(X)\to G(X)$ is a morphism in $\mathcal{C'}$ for any object $X$ in   $\mathcal{C}$,
\item For every morphism $f\colon X\to Y$ in $\mathcal{C}$ the following equality holds
\begin{equation*}
\eta_Y\circ F(f)=G(f)\circ\eta_X.
\end{equation*}
\end{enumerate}

It is well known that there exists a natural transformation of the identity functor to the double dual of a vector space functor. It turns out that the same holds for the functor $\normalfont\textrm{Dir}$.

\begin{proposition}
There exists a natural transformation of the identity functor on $\normalfont\textbf{Ban}$ to the functor $\normalfont\textrm{Dir}$.
\end{proposition}
\begin{proof}
Let $X$ be a Banach space. We define $\eta_X\colon X\to \displaystyle\lim_{\longrightarrow} X_n$ by
\begin{equation*}
\eta_X(x):= [(x,0)].
\end{equation*}
For any Banach space $X$ the mapping $\eta_X$ is an isometry. It suffices to show that given any $T\in B(X,Y)$ the diagram below commutes.
\[
\begin{tikzcd}[row sep=4em, column sep=4em]
\displaystyle\lim_{\longrightarrow} X_n \arrow[r, "\displaystyle\lim_{\longrightarrow} \mathcal{J}_{n}(T)"] & \displaystyle\lim_{\longrightarrow} Y_n \\
X \arrow[r, "T"'] \arrow[u, "\eta_X"] & Y \arrow[u, "\eta_Y"']
\end{tikzcd}
\]
For $x\in X$ the following equalities hold
\begin{equation*}
\begin{split}
\left(\lim_{\longrightarrow} \mathcal{J}_{n}(T)\circ\eta_X\right)(x)=\lim_{\longrightarrow} \mathcal{J}_{n}(T)([(x,0)])= \\ =[(T(x),0)]=\eta_Y(T(x))=\left(\eta_Y\circ f\right)(x).
\end{split}
\end{equation*}
Hence the result follows.
\end{proof}

We also have an inverse system $\langle X_n^*,\kappa_{n,m}^*\rangle$. Its inverse limit is the set
\begin{equation*}
\lim_{\longleftarrow} X_n^*=\left\{\phi=(\phi_n)_{n\in\mathbb{N}}\in\prod_{n=0}^{\infty}X_n^*\colon \phi_n=\kappa_{n,m}^*(\phi_m), \ \ n\le m, \ \Vert \phi\Vert=\sup_{n}\Vert \phi_n\Vert <\infty\right\}.
\end{equation*}

It should be noted that for non-reflexive $X$ the morphisms $\kappa_{n,m}^*$ are not isometric, but they are contractive instead.

We define the mapping $\displaystyle\lim_{\longleftarrow} \mathcal{J}_{n}(T)^*\colon\lim_{\longleftarrow} Y_n^*\to\lim_{\longleftarrow} X_n^*$ as follows
\begin{equation*}
\lim_{\longleftarrow} \mathcal{J}_{n}(T)^*\left(\phi\right)(n):=\mathcal{J}_{n}(T)^*\left(\phi_n\right) \ \textrm{for} \ \phi=(\phi_n)_{n\in\mathbb{N}}\in\lim_{\longleftarrow} Y_n^*.
\end{equation*}

Finally, we define the functor $\textrm{Inv}\colon\normalfont\textbf{Ban}\to\normalfont\textbf{Ban}$ by
\begin{equation*}
\textrm{Inv}\colon X\mapsto\lim_{\longleftarrow} X_n^*,
\end{equation*}
\begin{equation*}
\textrm{Inv}\colon T\mapsto\lim_{\longleftarrow} \mathcal{J}_{n}(T)^*.
\end{equation*}
\begin{proposition}
The functor $\normalfont\textrm{Inv}$ is a contravariant functor from $\normalfont\textbf{Ban}$ to $\normalfont\textbf{Ban}$.
\end{proposition}
\begin{proof}
Take any $T\in B(X,Y)$ and $S\in B(Y,Z)$ . For $\phi=(\phi_n)_{n\in\mathbb{N}}\in\displaystyle\lim_{\longleftarrow} Z_n^*$ the following equalities hold
\begin{equation*}
\begin{split}
\lim_{\longleftarrow} \left(\mathcal{J}_{n}(S)\circ \mathcal{J}_{n}(T)\right)^*\left(\phi\right)(n)=\left(\mathcal{J}_{n}(S)\circ \mathcal{J}_{n}(T)\right)^*(\phi_n)= \\ =\left(\mathcal{J}_{n}(T)^*\circ \mathcal{J}_{n}(S)^*\right)(\phi_n)=\mathcal{J}_{n}(T)^*\left(\lim_{\longleftarrow} \mathcal{J}_{n}(S)^*\left(\phi\right)(n)\right)= \\ =\lim_{\longleftarrow} \mathcal{J}_{n}(T)^* \left(\lim_{\longleftarrow} \mathcal{J}_{n}(S)^*\left(\phi\right)(n)\right).
\end{split}
\end{equation*}
Whence we obtain $\textrm{Inv}(S\circ T)=\textrm{Inv}(T)\circ \textrm{Inv}(S)$.

It is clear that the mapping $\textrm{Inv}(T)$ is linear. We have to prove that $\textrm{Inv}(T)$ is a bounded operator. For $\phi=(\phi_n)_{n\in\mathbb{N}}\in\displaystyle\lim_{\longleftarrow} Y_n^*$ we have
\begin{equation*}
\Vert\textrm{Inv}(T)(\phi)(n)\Vert=\Vert \mathcal{J}_{n}(T)^*(\phi_n)\Vert\le\Vert \mathcal{J}_{n}(T)^*\Vert\cdot\Vert \phi_n\Vert\le\Vert T\Vert\cdot\Vert\phi\Vert.
\end{equation*}
Hence $\Vert\textrm{Inv}(T)\Vert\le\Vert T\Vert$.

Finally, for any $\phi=(\phi_n)_{n\in\mathbb{N}}\in\displaystyle\lim_{\longleftarrow} X_n^*$ we have
\begin{equation*}
\begin{split}
\textrm{Inv}(\textrm{id}_X)(\phi)(n)=\lim_{\longleftarrow} \mathcal{J}_{n}(\textrm{id}_X)^*(\phi)(n)=\mathcal{J}_{n}(\textrm{id}_X)^*(\phi_n)= \\ =\textrm{id}_{X^*_n}(\phi_n)=\phi_n=\phi(n)=\textrm{id}_{\textrm{Inv}(X)}(\phi)(n).
\end{split}
\end{equation*}
This shows that the functor $\textrm{Inv}$ preserves an identity mapping.

\end{proof}
We have the following relation between these two functors.
\begin{theorem}
The functor $\normalfont\textrm{Inv}$ is adjoint to $\normalfont\textrm{Dir}$ in the following sense:
\begin{enumerate}[label={\textup{(\roman*)}}, widest=iii, leftmargin=*]
    \item $\normalfont\textrm{Dir}(X)^*=\textrm{Inv}(X)$ for any Banach space $X$,
    \item $\normalfont\textrm{Dir}(T)^*=\textrm{Inv}(T)$ for any $T\in B(X,Y)$.
\end{enumerate}
\end{theorem}
\begin{proof}
To prove the first part of this theorem we have to show that
\begin{equation*}
\left(\lim_{\longrightarrow} X_n\right)^*=\lim_{\longleftarrow} X_n^*,
\end{equation*}
up to an isometric isomorphism. Take any $\phi\in\displaystyle\left(\lim_{\longrightarrow} X_n\right)^*$. Define $\phi_n:=\phi\circ\kappa_{n,\infty}$, where $\kappa_{n,\infty}\colon X_n\to\displaystyle\lim_{\longrightarrow} X_n$ is the canonical mapping sending each element to the corresponding equivalence class. We claim that $\Psi(\phi):=(\phi_n)_{n\in\mathbb{N}}$ is an element of $\displaystyle\lim_{\longleftarrow} X_n^*$. For any $\gamma\in X_n$ the following equalities hold
\begin{equation*}
\begin{split}
\langle\kappa_{n,n+1}^*(\phi_{n+1}), \gamma\rangle=\langle\phi_{n+1}, \kappa_{n,n+1}(\gamma)\rangle=\langle\phi\circ\kappa_{n+1,\infty}, \kappa_{n,n+1}(\gamma)\rangle= \\
=\langle\phi, \left(\kappa_{n+1,\infty}\circ\kappa_{n,n+1}\right)(\gamma)\rangle=\langle\phi,  \kappa_{n,\infty}(\gamma)\rangle=\langle\phi\circ\kappa_{n,\infty}, \gamma\rangle=\langle\phi_n, \gamma\rangle.
\end{split}
\end{equation*}
Hence $\phi_n=\kappa_{n,n+1}^*\left(\phi_{n+1}\right)$ for any $n\in\mathbb{N}$. For arbitrary $n\le m$ we have
\begin{equation*}
\begin{split}
\kappa_{n,m}^*(\phi_m)=\left(\kappa_{m-1,m}\circ\ldots\circ\kappa_{n,n+1}\right)^*(\phi_m)=\left(\kappa_{n,n+1}^*\circ\ldots\circ\kappa_{m-1,m}^*\right)(\phi_m)= \\
=\left(\kappa_{n,n+1}^*\circ\ldots\circ\kappa_{m-2,m-1}^*\right)(\phi_{m-1})=\ldots =\phi_n.
\end{split}
\end{equation*}
Also for any $n\in\mathbb{N}$ the following inequalities are satisfied
\begin{equation*}
\Vert\phi_n\Vert=\Vert\phi\circ\kappa_{n,\infty}\Vert\le\Vert\phi\Vert\cdot\Vert\kappa_{n,\infty}\Vert\le\Vert\phi\Vert,
\end{equation*}
so that $\Vert\Psi (\phi)\Vert\le\Vert\phi\Vert$. Whence $\Psi\in\displaystyle\lim_{\longleftarrow} X_n^*$.

Now take any $(\psi_n)_{n\in\mathbb{N}}\in\displaystyle\lim_{\longleftarrow} X_n^*$. Define the mapping $\Psi(\psi)$ first on a a dense subset of $\displaystyle\lim_{\longrightarrow} X_n$ by
\begin{equation}
\Psi(\psi)\left([(x_n,n)]\right):=\psi_n(x_n).
\end{equation}
We prove that this mapping is well defined. Take $x_n\in X_n$ and $x_m\in X_m$ with $n\le m$ such that $x_n\sim x_m$. Hence $x_m=\kappa_{n,m}(x_n)$ and
\begin{equation*}
\langle\psi_m, x_m\rangle=\langle\psi_m, \kappa_{n,m}(x_n)\rangle=\langle\kappa_{n,m}^*(\psi_m),x_n\rangle=\langle\psi_n, x_n\rangle.
\end{equation*}
This shows that the mapping (2) is well defined. We prove that $\Psi(\psi)$ is a bounded linear functional. For any $x_n\in X_n$, $x_m\in X_m$ with $n\le m$ and $\alpha\in\mathbb{C}$ we have
\begin{equation*}
\Psi(\psi)\left(\alpha\cdot[(x_n,n)]\right)=\Psi(\psi)\left([(\alpha x_n,n)]\right)=\psi_n(\alpha x_n)=\alpha\psi_n(x_n)=\alpha\Psi(\psi)\left([(x_n,n)]\right),
\end{equation*}
also there exists $k\ge m\ge n$ such that
\begin{equation*}
\begin{split}
\Phi\left([(x_n,n)]+[(x_m,m)]\right)=\Psi(\psi)\left([(\kappa_{n,k}(x_n)+\kappa_{m,k}(x_m),k)]\right)= \\
=\psi_k(\kappa_{n,k}(x_n)+\kappa_{m,k}(x_m))=\left(\psi_k\circ\kappa_{n,k}\right)(x_n)+\left(\psi_k\circ\kappa_{m,k}\right)(x_m)= \\
=\left(\kappa_{n,k}^*(\psi_k)\right)(x_n)+\left(\kappa_{m,k}^*(\psi_k)\right)(x_m)=\psi_n(x_n)+\psi_m(x_m)= \\
=\Psi(\psi)\left([(x_n,n)]\right)+\Psi(\psi)\left([(x_m,m)]\right).
\end{split}
\end{equation*}
Hence $\Psi(\psi)$ is linear. For $x_n\in X_n$ we have
\begin{equation*}
\Vert\Psi(\psi)\left([(x_n,n)]\right)\Vert=\Vert\psi_n(x_n)\Vert\le\Vert\psi_n\Vert\cdot\Vert x_n\Vert\le\Vert\psi\Vert\cdot\Vert x_n\Vert.
\end{equation*}
Whence $\Psi(\psi)$ is a bounded linear functional. By continuity we may extend $\Psi(\psi)$ to the entire space $\displaystyle\lim_{\longrightarrow} X_n$.

The mappings $\displaystyle\Phi\colon\left(\lim_{\longrightarrow} X_n\right)^*\ni\phi\mapsto\Phi(\phi)\in\lim_{\longleftarrow} X_n^*$ and $\displaystyle\Psi\colon\lim_{\longleftarrow} X_n^*\ni\psi\mapsto\Psi(\psi)\in\left(\lim_{\longrightarrow} X_n\right)^*$
are inverse of each other and are contractions hence the spaces are isometrically isomorphic.

For the second claim of the theorem we need to verify that
\begin{equation*}
\left(\lim_{\longrightarrow} \mathcal{J}_{n}(T)\right)^*=\lim_{\longleftarrow} \mathcal{J}_{n}(T)^*.
\end{equation*}
Take arbitrary $\phi\in\displaystyle\left(\lim_{\longrightarrow} Y_n\right)^*$ and $x_n\in X_n$. Hence
\begin{equation*}
\begin{split}
\left\langle\left(\lim_{\longrightarrow} \mathcal{J}_{n}(T)\right)^*(\phi), [(x_n,n)]\right\rangle =\langle\phi, \lim_{\longrightarrow} \mathcal{J}_{n}(T)\left([(x_n,n)]\right)\rangle = \\ =\langle\phi, [(\mathcal{J}_{n}(T)(x_n),n)]\rangle
=\langle\phi\circ\kappa_{n,\infty}, \mathcal{J}_{n}(T)(x_n)\rangle = \\ =\langle \mathcal{J}_{n}(T)^*(\phi\circ\kappa_{n,\infty}), x_n \rangle =\left\langle\lim_{\longleftarrow} \mathcal{J}_{n}(T)^*(\phi\circ\kappa_{n,\infty}),x_n\right\rangle.
\end{split}
\end{equation*}
It follows from the identification $\displaystyle\left(\lim_{\longrightarrow} Y_n\right)^*=\lim_{\longleftarrow} Y_n^*$ that an element $\phi$ is identified with the sequence $(\phi\circ\kappa_{n,\infty})_{n\in\mathbb{N}}$ and we obtain the required equality.  The second part of the theorem is proved.
\end{proof}

The following fact is an easy consequence of previous theorems.

\begin{corollary}
Let $X$ be a Banah space. If $X=U\oplus V$ for some closed subspaces $U$, $V$ of $X$ then $\normalfont\textrm{Inv}(X)=\textrm{Inv}(U)\oplus\textrm{Inv}(V)$.
\end{corollary}
\begin{proof}
From Corollary 2.6 we know that $\textrm{Dir}(X)=\textrm{Dir}(U)\oplus\textrm{Dir}(V)$. It follows that $\textrm{Dir}(X)^{*}=\textrm{Dir}(U)^{*}\oplus\textrm{Dir}(V)^{*}$. Applying part (i) of Theorem 2.9 to the last equality finishes the proof.
\end{proof}

In the next result we shall prove that the action of our functors preserves the spectrum of an operator.

\begin{theorem}
Let $X$ be a Banach space. For any operator $T\in B(X)$ the following equalities hold
\begin{equation*}
\normalfont\sigma(T)=\sigma(\textrm{Dir}(T))=\sigma(\textrm{Inv}(T)),
\end{equation*}
where $\sigma(T)$ denotes the spectrum of $T$.
\end{theorem}
\begin{proof}
Suppose that $0\notin\sigma(T)$ and denote by $T^{-1}\in\mathcal{B}(X)$ an inverse of $T$. From the properties of functor $\textrm{Dir}$ we obtain
\begin{equation*}
\textrm{Dir}(TT^{-1})=\textrm{Dir}(T)\textrm{Dir}(T^{-1})=\textrm{Dir}(\textrm{id}_X)=\textrm{id}_{\textrm{Dir}(X)}.
\end{equation*}
It follows that $0\notin\sigma(\textrm{Dir}(T))$ and consequently $\sigma(\textrm{Dir}(T))\subset\sigma(T)$.

In order to prove the second inclusion assume that $0\in\sigma(T)$. There are two possibilities: either $0$ belongs to the approximate point spetrum of $T$ or the range $\mathcal{R}(T)$ of $T$ isn't dense in $X$. In the first case there exists a sequence $(x_k)$ of elements of $X$ such that $\Vert x_k\Vert =1$ and $Tx_k\to 0$. Hence
\begin{equation*}
\textrm{Dir}(T)([(x_k,0)])=[(Tx_k,0)]\to [(0,0)]
\end{equation*}
and $0$ is in the approximate point spectrum of $\textrm{Dir}(T)$.

Suppose that $\overline{\mathcal{R}(T)}\ne X$. By Hahn-Banach theorem there exists a non-zero functional $\phi\in X^{*}$ vanishing on $\mathcal{R}(T)$. Let $\phi_n:=\kappa_{0,n}(\phi)$ so that $\phi_n\in X_n^{*}$. We define the functional $\Phi$ on a dense subset of $\textrm{Dir}(X)$ by
\begin{equation*}
\Phi([(x_n,n)])=\phi_n(x_n).
\end{equation*}
By continunity $\Phi$ has a unique extension to $\Phi\in\textrm{Dir}(X)^{*}$. We will prove that $\Phi$ belongs to the kernel $\mathcal{N}(\textrm{Dir}(T)^{*})$ of $\textrm{Dir}(T)^{*}$. For any $x_n\in X_n$ we have
\begin{equation}
\begin{split}
\langle\textrm{Dir}(T)^{*}\circ\Phi , [(x_n,n)] \rangle =\langle\Phi , \textrm{Dir}(T)([(x_n,n)]) \rangle = \\ =\langle\Phi , [(\mathcal{J}_n(T) x_n,n)]\rangle =\langle \phi_n , \mathcal{J}_n(T) x_n \rangle.
\end{split}
\end{equation}
We claim that each $\phi_n$ vanishes on $\mathcal{R}(\mathcal{J}_n(T))$. It sufficies to check if this property holds for $n=1$. For any $x\in X$ we have
\begin{equation*}
\begin{split}
\langle\phi_1, T^{**}(\kappa_{0,1}(x))\rangle =\langle\kappa_{0,1}(\phi), T^{**}(\kappa_{0,1}(x))\rangle = \langle T^{**}(\kappa_{0,1}(x)), \phi \rangle = \\ = \langle\kappa_{0,1}(x), T^{*}\phi \rangle = \langle T^{*}\phi , x \rangle =  \langle \phi , Tx \rangle = 0.
\end{split}
\end{equation*}
By Goldstine theorem the canonical image of $X$ is weak-* dense in $X^{**}$. Since $T^{**}$ is weak-* continuous it follows that $\phi_1$ vanishes on $\mathcal{R}(T^{**})$. Consequently, from (3) we deduce that $\mathcal{N}(\textrm{Dir}(T)^{*})$ is non-zero or equivalently $\overline{\mathcal{R}(\textrm{Dir}(T))}\ne\textrm{Dir}(X)$. Hence $\sigma(T)=\sigma(\textrm{Dir}(T))$. The equality $\sigma(\textrm{Dir}(T))=\sigma(\textrm{Inv}(T))$ follows from the fact that $\textrm{Inv}(T)$ is adjoint to $\textrm{Dir}(T)$.
\end{proof}

Our next result shows that for a Fredholm operator $T\in B(X)$ the action of our functors yields also a Fredholm operator. Moreover, corresponding equalities of indices are satisfied. Let us recall that for a Fredholm operator $T$ its index, denoted by $\textrm{i}(T)$, is the difference between the dimension of kernel of $T$ and the dimension of cokernel of $T$. Then $T^*\in B(X^{*})$ is also  Fredholm   and $\textrm{i}(T^{*})=-\textrm{i}(T)$.

\begin{theorem}
Let $X$ be a Banach space possessing an approximation property. If $T\in B(X)$ is a Fredholm operator then $\normalfont\textrm{Dir}(T)$ and $\normalfont\textrm{Inv}(T)$ are also Fredholm operators. What is more, the equalities $\normalfont\textrm{i}(T)=\textrm{i}(\textrm{Dir}(T))=-\textrm{i}(\textrm{Inv}(T))$ are satisfied.
\end{theorem}
\begin{proof}
Let $T\in B(X)$ be a Fredholm operator. By   \cite[Lemma 4.39]{A} there exists a closed subspace $V$ and a finite dimensional subspace $W$ such that
\begin{equation*}
    T=0\oplus S\colon\mathcal{N}(T)\oplus V\to W\oplus\mathcal{R}(T),
\end{equation*}
where $0$ denotes the operator constantly equal $0$ and $S:=T\big|_V$  is an isomorphism. From Theorem 2.6 we obtain that
\begin{equation*}
    \textrm{Dir}(T)=0\oplus\textrm{Dir}(S)\colon\textrm{Dir}(\mathcal{N}(T))\oplus\textrm{Dir}(V)\to\textrm{Dir}(W)\oplus\textrm{Dir}(\mathcal{R}(T)).
\end{equation*}
From Proposition 2.2 operator $\textrm{Dir}(S)$ is also an isomorphism and we get the equalities $\mathcal{N}(\textrm{Dir}(T))=\textrm{Dir}(\mathcal{N}(T))$ and $\mathcal{R}(\textrm{Dir}(T))=\textrm{Dir}(\mathcal{R}(T))$. Hence from the fact that $\mathcal{N}(T)$ and $W$ are finite dimensional we have that
\begin{equation*}
    \dim\mathcal{N}(\textrm{Dir}(T))=\dim\textrm{Dir}(\mathcal{N}(T))=\dim\mathcal{N}(T)<\infty
\end{equation*}
and
\begin{equation*}
\begin{split}
    \dim\textrm{Dir}(X)\big/\mathcal{R}(\textrm{Dir}(T))=\dim\textrm{Dir}(X)\big/\textrm{Dir}(\mathcal{R}(T))= \\
    =\dim\textrm{Dir}(W)=\dim W=\dim X\big/\mathcal{R}(T)<\infty.
\end{split}
\end{equation*}
From above equalities and from part (ii) of Theorem 2.9 we obtain that $\textrm{Dir}(T)$ and $\textrm{Inv}(T)$ are Fredholm operators and $\textrm{i}(T)=\textrm{i}(\textrm{Dir}(T))=-\textrm{i}(\textrm{Inv}(T))$.
\end{proof}

In the last theorem of this section we are going to prove that   our functors preserve compactness of operators. Our method of proof requires that the Banach space $X$ has an \textit{approximation property} so that any compact operator on $X$ is a norm-limit of finite rank operators.

\begin{theorem}
Let $X$ be a Banach space possessing an approximation property. If $T\in B(X)$ is a compact operator then $\normalfont\textrm{Dir}(T)$ and $\normalfont\textrm{Inv}(T)$ are also compact operators.
\end{theorem}
\begin{proof}
By the assumption $T$ is a limit of finite rank operators. Since $\Vert T\Vert=\Vert\textrm{Dir}(T) \Vert$ it suffices to prove that rank of $\textrm{Dir}(T)$ equals the rank of $T$, even in the case when this rank equals one. Indeed, finite rank operators are sums of rank one operators and $\textrm{Dir}$ preserves addition of operators. In the later case $T$ has the form $Tx=\varphi(x)z$ for $\varphi\in X^{*}$ and $z\in X$. For $\gamma\in X_n$ we have
\begin{equation*}
    \mathcal{J}_{n}(T)(\gamma)=\kappa_{0,n}(\varphi)(\gamma)\kappa_{0,n}(z).
\end{equation*}
Hence
\begin{equation*}
\begin{split}
    \textrm{Dir}(T)(\kappa_{n,\infty}(\gamma))=\kappa_{n,\infty}(\mathcal{J}_{n}(T)(\gamma))=\kappa_{n,\infty}(\kappa_{0,n}(\varphi)(\gamma)\kappa_{0,n}(z))= \\ =\kappa_{0,n}(\varphi)(\gamma)(\kappa_{n,\infty}\circ\kappa_{0,n})(z)=\kappa_{0,n}(\varphi)(\gamma)\kappa_{0,\infty}(z)=\Phi(\gamma)\kappa_{0,\infty}(z),
    \end{split}
\end{equation*}
where $\Phi$ is a functional $\displaystyle\lim_{\longleftarrow}\kappa_{0,n}(\varphi)$. By the density of $\bigcup\kappa_{n,\infty}(X_n)$ in $\textrm{Dir}(X)$ it follows that $\textrm{Dir}(T)$ is a rank one operator. From Schauder theorem we deduce that $\textrm{Inv}(T)$ is also a compact operator.
\end{proof}

We conjecture that the approximation property assumption can be omitted.

\begin{hypothesis}
For any Banach space $X$ if $T\in B(X)$ is a compact operator then $\normalfont\textrm{Dir}(T)$ is a compact operator.
\end{hypothesis}

\section{Supports of representing measures}

In this section we outline an application of our functors to uniform algebras.  Here  the functor  Dir assigns to a uniform algebra another function algebra by extending Arens product.  It should be noted that  the inverse system consists of spectra of bidual uniform algebras, we restrict Inv to multiplicative linear functionals, sending the unit element to 1. This "restricted Inv" will act as a functor in the category of compact Hausdorff spaces.

Let $K$ be a compact Hausdorff space. By a \textit{uniform algebra} on $K$ we understand a closed unital subalgebra $A$ of $C(K)$ separating the points of $K$. An important example of a uniform algebra is $A(G)$, the algebra of those analytic functions on a strictly pseudoconvex domain $G\subset\mathbb{C}^d$, which have continuous extensions to the Euclidean closure $\overline G$.

The \textit{spectrum} of $A$, denoted by $\textrm{Sp}(A)$, is the set of all nonzero multiplicative and linear functionals on $A$. Endowed with the Gelfand (=weak-*) topology,  $\textrm{Sp}(A)$ is a compact Hausdorff space, containing a homeomorphic copy of $K$. The  natural embedding of $K$ into the spectrum is given by    $K\ni x\mapsto\delta_x\in\textrm{Sp}(A)$, where $\delta_x(f)=f(x)$ for $f\in A$.  
A uniform algebra $A$ on $K$ is called a \textit{natural uniform algebra} if $\textrm{Sp}(A)=K$ in the sense of this embedding. It is known that $A(G)$ is natural on $G$ if the domain $G$ is strictly pseudoconvex (see \cite{HS}).

Let $A$ be a Banach algebra. For $\lambda\in A^{*}$, define $a\cdot\lambda$ and $\lambda\cdot a$ by duality
\begin{equation*}
\langle a\cdot\lambda, b\rangle:=\langle\lambda, ba\rangle, \quad \langle\lambda\cdot a,b\rangle:=\langle\lambda, ab\rangle, \quad a,b\in A.
\end{equation*}
Now, for $\lambda\in A^{*}$ and $M\in A^{**}$, define $\lambda\cdot M$ and $M\cdot\lambda$ by
\begin{equation*}
\langle \lambda\cdot M,a\rangle:=\langle M,a\cdot\lambda\rangle, \quad \langle M\cdot\lambda ,a\rangle:=\langle M, \lambda\cdot a\rangle, \quad a\in A.
\end{equation*}
Finally, for $M,N\in A^{**}$, define
\begin{equation*}
\langle M\Box N,\lambda\rangle:=\langle M,N\cdot\lambda\rangle, \quad \langle M\Diamond N,\lambda\rangle:=\langle N,\lambda\cdot M\rangle, \quad \lambda\in A^{*}.
\end{equation*}

The products $\Box$ and $\Diamond$ are called, respectively, the \textit{first} and \textit{second Arens products} on $A^{**}$. A bidual of $A$ is Banach algebra with respect to Arens products. The natural embedding of $A$ into its bidual identifies $A$ as a norm -- closed subalgebra of both $(A^{**},\Box)$ and $(A^{**},\Diamond)$.

All $C^{*}$-algebras are \emph{Arens regular} in the sense, that the two products $\Box$ and $\Diamond$ agree on $A^{**}$. Closed subalgebras of Arens regular algebras are Arens regular, hence all uniform algebras are Arens regular. Also the bidual of uniform algebra is again  a uniform algebra with respect to the Arens products  (cf.\cite{Da} and \cite{D}). %For further information about Arens products we refer to \cite{Da} and \cite{D}.

 There is an equivalence relation on the spectrum  of a uniform algebra given by
\begin{equation*}
\Vert\phi -\psi\Vert <2,
\end{equation*}
with $\Vert\cdot\Vert$ denoting the norm in $A^*$ for $\phi$, $\psi\in\textrm{Sp}(A)$. The equivalence classes under the above relation are called \textit{Gleason parts}.

The space of complex, regular Borel measures on $K$, with total variation norm will be denoted by $M(K)$. As the consequence of Riesz-Markov-Kakutani Representation Theorem we have $M(K)=C(K)^*$.

We say that $\mu\in M(K)$ is \textit{representing}   \textit{measure} for $\phi\in\textrm{Sp}(A)$ if $\mu $ is probabilistic and

%A \textit{representing}   \textit{measure} for $\phi\in\textrm{Sp}(A)$ is a probabilistic   measure $\mu\in M(K)$ such that
\begin{equation*}
\phi(f)=\int_Kfd\mu \quad \textrm{for any} \quad f\in A.
\end{equation*}

The set of all representing measures for $\phi\in\textrm{Sp}(A)$ is denoted by $M_{\phi}(K)$. By \cite[Proposition 8.2]{Co} for any functional $\phi\in\textrm{Sp}(A)$ there exists at least one representing measure. In fact this measure can be chosen with its support contained in the Shilov boundary of $A$. We refer to \cite{Co}, \cite{G} and \cite{S} for further information on uniform algebras.

The idea of using second duals in studying the spectrum of $H^\infty(G)$ is based on the observation that this algebra can be seen either as a weak-* closure of a (much easier to study) algebra $A:=A(G)$ in  $A^{**}$  or as a quotient algebra  of $A^{**}$ by one of its  ideals. Moreover, to any Gleason part $\gamma $ of $\textrm{Sp}(A)$ there corresponds an idempotent $g\in A^{**}$ vanishing on $\textrm{Sp}(A)\setminus \gamma$ and equal 1 on $
\gamma$ (also $g=1$ on its w-* closure $\overline{\gamma}$ in $\textrm{Sp}(A^{**})$). But the behaviour of $g$ outside $\overline{\gamma}$ is hard to control.

It would be more convenient to have the idempotent related to $\gamma$ in the same algebra $A$ (or in the canonical image of $A$ in $A^{**}$). Unfortunately, this is impossible. But getting close to such a situation  is offered by direct limits, where the passage from $n$ to $ n+1$  is in a sense reminiscent to the "Hilbert hotel method".

Let $A$ be a natural uniform algebra on $K$. We define $A_0:=A$, $A_{n+1}:=A_n^{**}$ for $n=0,1,2,\ldots $. By $\mathcal{A}$ we denote the direct limit of the direct system $\langle A_n, \kappa_{n,m}\rangle$. Let $f\in A_n$ and $h\in A_m$ for some $n\le m$. The multiplication on $\mathcal{A}$ is given by
\begin{equation*}
[(f,n)]\cdot [(h,m)]:=[(\kappa_{n,m}(f)\cdot h,m)].
\end{equation*}
Since $\langle A_n, \kappa_{n,m}\rangle$ is an isometric direct system of uniform algebras its direct limit is again a uniform algebra.

%Before we state the result concerning supports of representing measures we need to introduce some definitions. Let $A$ be a natural uniform algebra on $K$. 

For $\mu\in M(\textrm{Sp}(A))$, where $A$ is a natural uniform algebra on $K$, we define a measure $k(\mu)\in M(\textrm{Sp}(C(K)^{**}))$ by duality
\begin{equation*}
\langle k(\mu), f\rangle := \langle f, \mu\rangle, \quad f\in C(\textrm{Sp}(C(K)^{**})).
\end{equation*}
In this definition we use an identification $M(K)^*=C(\textrm{Sp}(C(K)^{**})$. The identification is valid because the second dual of $C(K)$ endowed with Arens product is a commutative $C^*$-algebra. Hence by Gelfand-Najmark theorem it is isometrically isomorphic to $C(\textrm{Sp}(C(K)^{**}))$.

Let $\phi\in\textrm{Sp}(A)$. The mapping $k\colon\textrm{Sp}(A)\to\textrm{Sp}(C(K)^{**})$ is given by

\begin{equation*}
k(\phi)(F)=\int Fdk(\nu ),
\end{equation*}
where $F\in C(K)^{**}$ and $\nu\in M_{\phi}(K)$. It can be easily shown that the function $k$ does not depend on the choice of the representing measure $\nu$.

On the set $\textrm{Sp}(C(K)^{**})$ we introduce an equivalence relation $\simeq$ as follows
\begin{equation*}
x\simeq y \ \textrm{if and only if} \ f(x)=f(y) \ \textrm{for all} \ f\in A^{**}.
\end{equation*}
The mapping $\Pi\colon\textrm{Sp}(C(K)^{**})\to\textrm{Sp}(C(K)^{**})\big/_\simeq\subset\textrm{Sp}(A^{**})$ denotes the canonical surjection assigning to each element of $\textrm{Sp}(C(K)^{**})$ its equivalence class. Finally, we define $j\colon\textrm{Sp}(A)\to\textrm{Sp}(A^{**})$ by $j:=\Pi\circ k$.

Extending previous definitions to the $n$-th level spaces we obtain functions $j_n\colon\textrm{Sp}(A_n)\to\textrm{Sp}(A_{n+1})$. For $n\le m$ we define $j_{n,m}\colon \textrm{Sp}(A_n)\to \textrm{Sp}(A_m)$ by
\begin{equation*}
j_{n,m}:=j_{m-1}\circ j_{m-2}\circ\dots\circ j_{n}, \ \textrm{if} \ n<m.
\end{equation*}
and $j_{n,n}$ is defined to be the identity mapping. For $x\in\textrm{Sp}(A_n)$ denote by $j_{n,\infty}(x)$ embedding into the inverse limit $\displaystyle\lim_{\longleftarrow}\textrm{Sp}(A_n)$, defined by
\begin{equation*}
j_{n,\infty}(x):=(\kappa_{0,n}^*(x),\ldots,\kappa_{n-1,n}^*(x),x,j_{n,n+1}(x),j_{n,n+2}(x),\ldots).
\end{equation*}
It should be noted that the inverse limit of the spectrum is limit in the category of compact Hausdorff spaces. In this case part (i) of Theorem 2.9 becomes the equality
\begin{equation*}
\lim_{\longleftarrow}\textrm{Sp}(A_n)=\textrm{Sp}(\mathcal{A}).
\end{equation*}

The following assumption will be needed: there exists an open Gleason part $G$ satisfying
\begin{equation}
j_{0,n}(G)=(\kappa_{0,n}^*)^{-1}(G) \ \textrm{for} \ n=1,2,\ldots,\infty \tag{$\dagger$} \label{eq:special}.
\end{equation}
It is shown in \cite{KR} that this assumption is satisfied for $A=A(G)$ if $G\subset\mathbb{C}^d$, $d>1$, is a strictly pseudoconvex domain.

Let $G\subset\mathbb{C}^d$,  be a strictly pseudoconvex domain. By $H^{\infty}(G)$ we denote the Banach algebra of all bounded and analytic functions on $G$. The Corona theorem states that the set $G$ (identified with the set of evaluation functionals) is dense (in  Gelfand topology) in the spectrum of $H^{\infty}(G)$  the Banach algebra of all bounded   analytic functions on $G$. The following abstract result was a key to the proof of Corona theorem by M. Kosiek and K. Rudol.
\begin{theorem}[\cite{KR}]
Let $A$ be a natural uniform algebra on $K$. If $G$ is an open Gleason part in $\emph{Sp}(A)$ satisfying \textup{(}$\dagger $\textup{)} then support of every representing measure $\mu_0$ for $A_n$ at any point $x_0$ in $j_{0,n}(G)$ lies in the Gelfand closure of $j_{0,n}(G)$.
\end{theorem}

Proof of the last theorem is first carried for the direct limit algebra (the $n=\infty$ case) and then "projected" by the natural mappings to the finite level algebras (with $n<\infty$). The detailed exposition of this construction can be found in \cite{KR}.

One could ask if for arbitrary uniform algebra the support of any representing measure of a point in $G$ lies in the Gelfand closure of a Gleason part $G$. The answer to this question is no. The work \cite{C} of B. J. Cole assures an existence of natural uniform algebra $A$ on $K$ with a property that each point of $\textrm{Sp}(A)$ is a one-point Gleason part and the Shilov boundary of $A$ is not the entire set $K$. Since any point of $\textrm{Sp}(A)$ has a representing measure with its support contained in the Shilov boundary of $A$ the general result doesn't hold. It is worth noting that the construction of Cole's algebra also relies on the direct limit technique.

%\\
%The subject of the paper was suggested by author's supervisor, Professor Rudol (the coauthor of \cite{KR}). The author wishes to express his deepest gratitude to him, not only for conveying the topic but also for his great support and guidance without which this paper would have never been written.\\

\end{document}